\documentclass[reqno, 11pt]{amsart}
\usepackage{amsthm, amsmath, amssymb, graphicx,enumerate,paralist,bbm,mathrsfs,cite,xcolor}
\usepackage[marginratio=1:1,height=584pt,width=450pt,tmargin=117pt]{geometry}
\numberwithin{equation}{section}

\usepackage[pdfstartview=FitV, 
pagebackref=false]{hyperref}

\usepackage[color,final]{showkeys} 

\colorlet{refkey}{orange!20}
\colorlet{labelkey}{blue!30}
\definecolor{refkey}{gray}{.75}
\definecolor{labelkey}{gray}{.5}

\newtheorem{theorem}{Theorem}[section]
\newtheorem{proposition}[theorem]{Proposition}
\newtheorem{lemma}[theorem]{Lemma}
\newtheorem{claim}[theorem]{Claim}
\newtheorem{corollary}[theorem]{Corollary}

\newtheorem*{question*}{Question}
\newtheorem{remark}[theorem]{Remark}

\theoremstyle{definition}

\newtheorem*{definition*}{Definition}

\theoremstyle{remark}
\newtheorem*{remark*}{Remark}

\newcommand{\abs}[1]{\left\lvert#1\right\rvert}

\newcommand{\paren}[1]{\left( #1 \right)}

\renewcommand{\epsilon}{\varepsilon}

\newcommand{\x}{\times}

\newcommand{\e}{\epsilon}

\renewcommand{\P}{\mathbb{P}}
\newcommand{\sG}{\mathscr{G}}
\newcommand{\EE}{\mathbb{E}}

\newcommand{\RR}{\mathbb{R}}

\newcommand{\cW}{\mathcal{W}}

\newcommand{\cG}{\mathcal{G}}
\newcommand{\cP}{\mathcal{P}}

\newcommand{\Bin}{\operatorname{Bin}}

\author{Eyal Lubetzky}
\address{E.\ Lubetzky\hfill\break
Courant Institute of Mathematical Sciences,
New York University, New York, NY 10012, USA}
\email{eyal@courant.nyu.edu}

\author{Yufei Zhao}
\address{Y.\ Zhao\hfill\break
Mathematical Institute, University of Oxford,
Oxford OX2 6GG, United Kingdom}
\email{yufei.zhao@maths.ox.ac.uk}
\thanks{Y.\ Zhao was supported by a Microsoft Research Ph.D.\ Fellowship.}

\title[On upper tails in sparse random graphs]{On the variational problem for \\ upper tails in sparse random graphs}

\begin{document}

\begin{abstract}\vspace{-0.175cm}

What is the probability that the number of triangles in $\cG_{n,p}$, the Erd\H{o}s-R\'enyi random graph with edge density $p$,
is at least twice its mean? Writing it as $\exp[- r(n,p)]$, already the order of the rate function $r(n,p)$ was a longstanding open problem when $p=o(1)$, finally settled in 2012 by
Chatterjee and by DeMarco and Kahn, who independently showed that $r(n,p)\asymp n^2p^2 \log (1/p)$
for $p \gtrsim \frac{\log n}n$; the exact asymptotics of $r(n,p)$ remained unknown.

The following variational problem can be related to this large deviation question at $p\gtrsim \frac{\log n}n$: for $\delta>0$ fixed, what is the minimum  asymptotic $p$-relative entropy of a weighted graph on $n$ vertices with triangle density at least $(1+\delta)p^3$?
A beautiful large deviation framework of Chatterjee and Varadhan (2011) reduces upper tails for triangles to a limiting version of this problem for \emph{fixed} $p$. A very recent breakthrough of Chatterjee and Dembo extended its validity to $n^{-\alpha}\ll p \ll 1$ for an explicit $\alpha>0$, and plausibly it holds in all of the above sparse regime.

In this note we show that the solution to the variational problem is $\min\{\frac12 \delta^{2/3}\,,\, \frac13 \delta\}$ when $n^{-1/2}\ll p \ll 1$ vs.\ $\frac12 \delta^{2/3}$ when $n^{-1} \ll p\ll n^{-1/2}$ (the transition between these regimes is expressed in the count of triangles minus an edge in the minimizer).
From the results of Chatterjee and Dembo, this shows for instance that the probability that $\cG_{n,p}$ for $ n^{-\alpha} \leq p \ll 1$ has twice as many triangles as its expectation is
$\exp[-r(n,p)]$ where $r(n,p)\sim \frac13 n^2 p^2\log(1/p)$.
Our results further extend to $k$-cliques for any fixed $k$, as well as give the order of the upper tail rate function for an arbitrary fixed subgraph when $p\geq n^{-\alpha}$.
\end{abstract}

{\mbox{}
\vspace{-1.35cm}
\maketitle
}
\vspace{-0.65cm}
\section{Introduction}\label{sec:intro}

The following question regarding upper tails for triangle counts in $\cG_{n,p}$, the Erd\H{o}s-R\'enyi random graph with edge density $p$,
has been extensively studied, being a representing example of large deviations for subgraph counts in random graphs (see, e.g.,~\cite{JR02,Vu01,KV04,JOR04,JR04,Cha12,DK12b,DK12} as well as~\cite{Bol,JLR} and the references therein):
\begin{question*}
What is the probability that the number of triangles in $\cG_{n,p}$ is at least twice its mean, or more generally, larger by a factor of $1+\delta$ for $\delta>0$ fixed?
\end{question*}
In the dense case ($p$ fixed), the limiting asymptotics  of the rate function --- the normalized logarithm of this probability, here denoted by $r(n,p,\delta)$ ---
was reduced to an analytic variational problem on symmetric functions $f:[0,1]^2\to[0,1]$ (for a large class of large deviation questions) by Chatterjee and Varadhan~\cite{CV11}.
However, for $p=o(1)$, obtaining the order of $r(n,p,\delta)$ was already a longstanding open problem. That $n^2 p^2 \lesssim r(n,p,\delta)\lesssim n^2p^2\log(1/p)$ followed from the works of Vu~\cite{Vu01} and Kim and Vu~\cite{KV04} (see also~\cite{JOR04}), and this question was finally settled in 2012 by
Chatterjee~\cite{Cha12} and by DeMarco and Kahn~\cite{DK12}, where it was independently shown that $r(n,p,\delta)\asymp n^2p^2 \log (1/p)$
for $p \gtrsim \frac{\log n}n$
(see~\cite{Cha12,DK12} for an account of the rich related literature).
The exact asymptotics of this rate function was not known for any $\frac{\log n}n \lesssim p \ll 1$.

Note that for the dense regime of fixed $p$, while~\cite{CV11} provided a closed form for the rate function in terms of the above variational problem,
its solution is only known in a subset of the range of parameters $(p,\delta)$ known as the \emph{replica symmetric} phase (where the excess in the number of triangles is explained by encountering too many edges that are essentially uniformly distributed), and little is known on its complement (the \emph{symmetry breaking} phase; see our previous work~\cite{LZ} where this phase diagram was determined).

The variational problem in~\cite{CV11} can be viewed, via Szemer\'edi's regularity lemma~\cite{Sze78} and the theory of
graph limits by Lov\'asz et al.~\cite{LS06,LS07,BCLSV08}, as the limit of the following problem.

\begin{definition*}
  [Discrete variational problem for upper tails of triangles]
Let $\sG_n$ denote the set of weighted undirected graphs on $n$ vertices with
edge weights in $[0,1]$, i.e.,
\[ \sG_n = \Big\{ G = (g_{ij})_{1 \leq i < j \leq n} \;:\; 0 \leq g_{ij} \leq
1 ~,~g_{ij} = g_{ji}~,~g_{ii} = 0 ~\mbox{ for all $i,j$}\Big\}\,.\]
The variational problem for $\delta>0$ and $0<p<1$ is given by
\begin{equation} \label{eq:var}
\phi(n,p,\delta) := \inf \Big\{ I_p(G) : G \in \sG_n \text{ with }
t(G) \geq (1 + \delta)p^3\Big\}\,,
\end{equation}
where
\[
t(G) := n^{-3} \sum_{1 \leq i,j,k \leq n} g_{ij} g_{jk} g_{ik}
\]
is the density of (labeled) triangles in $G$, and $I_p(G)$ is its entropy relative to $p$, i.e.,
\[
I_p(G) := \sum_{1 \leq i < j \leq n} I_p(g_{ij})\quad\mbox{ with }\quad I_p(x) := x \log\frac{x}{p} + (1-x) \log\frac{1-x}{1-p}\,.
\]
\end{definition*}
Indeed, it follows from the powerful large deviation framework of~\cite{CV11} that for $p$ fixed (the dense regime)
$\frac{1}{n^2} \log \P\left(t(\cG_{n,p})\geq (1+\delta)p^3\right)$ tends as $n\to\infty$ to the limit of $-\phi(n,p,\delta)/n^2$.

However, in the sparse regime of $p=o(1)$, which lacks the rich set of tools that are based on Szemer\'edi's regularity lemma for dense graphs, there were no counterparts to this result until a very recent breakthrough by Chatterjee and Dembo~\cite{CD}. There it was shown that the discrete variational problem~\eqref{eq:var} does govern the rate function of subgraph counts as long as $p \geq n^{-\alpha}$ for a suitable constant $\alpha$. In particular, for triangle counts (see~\cite[Theorem~1.2]{CD} and the remark following it, yielding a slightly wider range than the one stated next) one has that
\begin{equation}
  \label{eq-cd}
   \P\left(t(\cG_{n,p}) \geq (1 +
    \delta)p^3\right) = \exp\left[-(1-o(1))\phi(n,p,\delta)\right]
\end{equation}
 whenever $n^{-1/42}\log n \leq p \ll 1$ (this should extend to smaller $p$, as commented in~\cite{CD}; in fact, it is plausible that this result holds throughout the sparse regime of $\frac{\log n}n \ll p \ll 1$.)

In this note we establish the following for the discrete variational problem~\eqref{eq:var}.
\begin{theorem}
  \label{thm:var}
  Fix $\delta > 0$. If $n^{-1/2}\ll p \ll 1$, then
  \begin{equation} \label{eq:var-ans} \lim_{n \to \infty}
    \frac{\phi(n, p, \delta)}{n^2 p^2 \log(1/p)} =
    \min\bigg\{\frac{\delta^{2/3}}{2}\;,\;\frac{\delta}{3}\bigg\}\,.
  \end{equation}
  On the other hand, if $n^{-1}  \ll p \ll n^{-1/2}$, then
  \begin{equation} \label{eq:var-ans-b} \lim_{n \to \infty}
    \frac{\phi(n, p, \delta)}{n^2 p^2 \log(1/p)} =
    \frac{\delta^{2/3}}{2}\,.
  \end{equation}
\end{theorem}
One can then deduce the following from the above result~\eqref{eq-cd} of Chatterjee and Dembo.
\begin{corollary}\label{cor:upper-tail}
For any $\delta>0$, if $  n^{-1/42} \log n\leq p \ll 1$ then
  \begin{equation*}
    \P\left( t(\cG_{n,p}) \geq (1 +
    \delta)p^3\right) = \exp\left[ - (1-o(1))\min\left\{\tfrac12 \delta^{2/3}\,,\, \tfrac13 \delta\right\}n^2 p^2 \log(1/p)\right]\,.
  \end{equation*}
\end{corollary}
The lower bound is explained by forcing either a set of
$k=\delta^{1/3} np$ vertices to be a clique (with probability $p^{\binom{k}2} = p^{(\delta^{2/3}/2+o(1))n^2 p^2}$) or a set of $\ell=\frac13 \delta np^2$ vertices to be
connected to all other vertices (with probability $p^{\ell(n-\ell)} = p^{(\delta/3+o(1))
  n^2p^2}$), the latter being preferable if and only if $\delta < 27/8$.

In fact, these constructions for the lower bound on $\P\left(t(\cG_{n,p})\geq (1+\delta)p^3\right)$ further explain the two separate regimes in Theorem~\ref{thm:var}. When $p\ll 1/\sqrt{n}$, the second (bipartite) construction --- involving $\ell \asymp np^2$ vertices --- ceases to be a viable option, as then we have $\ell=o(1)$.
As remarked next, this translates into a qualitative difference between the solutions of the variational problem in each of these regimes, expressed in terms of
\[
s(G) := n^{-3} \sum_{1 \leq i \leq n} \big ( \sum_{1\leq j \leq n} g_{ij}\big)^2\,,
\]
equivalent to the asymptotic density of triangles minus an edge (i.e., $K_{1,2}$ homomorphisms,
which in $\cG_{n,p}$ have average density $p^2$, and so an excess of $\frac13 \delta p^2$ in their density, of which a $p$-fraction forms triangles via an extra edge, translates to $\delta p^3$ additional labeled triangles).
\begin{remark}
  \label{rem:var-K12}
  The proof of Theorem~\ref{thm:var} shows that for any fixed $0<\delta<\frac{27}8$, if $G_n\in\sG_n$ is a sequence of weighted graphs satisfying $t(G_n)\geq (1+\delta) p^3$ and $I_p(G_n) \sim \phi(n,p,\delta)$
  then
\[
  \lim_{n\to\infty} \frac{s(G_n)}{p^2} = \begin{cases}
    1+\delta/3 & \mbox{if }n^{-1/2} \ll p \ll 1\,,\\
    1        & \mbox{if }n^{-1} \ll p \ll n^{-1/2}\,.
  \end{cases}
\]
For fixed $\delta > \frac{27}8$, the term $1+\delta/3$ in the first case ($n^{-1/2} \ll p \ll 1$) is replaced by $1$.
\end{remark}
Regarding the behavior when $p\asymp n^{-1/2}$, there one expects a similar structure: i.e., whenever the bipartite construction is preferable, the optimal solution should feature a large bipartite subgraph while adhering to the integrality restrictions. It is plausible that methods similar to those used in this work can establish the solution in that regime as well.

Our arguments extend to yield analogous results for $k$-clique counts,
where, for instance, the right-hand side of~\eqref{eq:var-ans} (giving the asymptotics of the rate function provided $n^{-\alpha'} \ll p \ll 1$ for $\alpha'(k)>0$)  is replaced by $\min\{\frac12\delta^{2/k},\delta/k\}$; see Theorem~\ref{thm:var-clq} and Corollary~\ref{cor:upper-tail-clq}.
For a general graph on $k$ vertices, the \emph{order} of the rate function at $p\geq n^{-\alpha''}$ is given by Corollary~\ref{cor:gen-H}.\footnote{In our follow-up work \cite{BGLZ} jointly with Bhattacharya and Ganguly, we extend this and find the asymptotic rate function for every graph $H$. The rate is given in terms of a certain independence polynomial, and exhibits a dichotomy with respect to $\delta$ if and only if $H$ is a regular graph. See~\cite{BGLZ} for the statements of these newer results.}

Finally, it is worthwhile mentioning that even without appealing to the new machinery of~\cite{CD}, if $p$ tends to 0 sufficiently slowly with $n$ --- namely, $(\log n)^{-1/6} \ll p \ll 1$ --- then
Eq.~\eqref{eq-cd} (stating that the variational problem~\eqref{eq:var} gives the asymptotic rate function for large deviations of triangles) follows essentially from the framework of Chatterjee and Varadhan~\cite{CV11} (and similarly for any fixed subgraph); instead of using the theory of graph limits or Szemer\'edi's regularity lemma, one can derive this statement by appealing in their framework to the \emph{weak} regularity lemma of Frieze and Kannan~\cite{FK99} (we include this reduction for completeness; see~\S\ref{sec:weak-ldp}).

\subsection*{Notation and organization}
On occasion we will write $f_n \lesssim g_n$ instead of $f_n=O(g_n)$ for brevity, as well as $f_n \ll g_n$ instead of $f_n=o(g_n)$ (similarly for $f_n \gtrsim g_n$ and $f_n \gg g_n$); we let $f_n \sim g_n$ denote $f_n = (1+o(1))g_n$, and $f\asymp g$ denotes $f_n \lesssim g_n \lesssim f_n$.

This paper is organized as follows. In~\S\ref{sec:cont-var} we give upper and lower bounds for the discrete variational problem~\eqref{eq:var}: the construction of a clique/bipartite subgraph, and a (relaxed) continuous variational problem, whose solution we denote by $\phi(\delta,p)$ (notice this variant no longer depends on $n$; see Eq.~\eqref{eq:var2} below). The analysis of the latter appears in~\S\ref{sec:solve-var}, and \S\ref{sec:cliques} contains the extension of these results to $k$-cliques for any fixed $k$. Finally, \S\ref{sec:weak-ldp} contains the reduction of the upper tail to the variational problem~\eqref{eq:var} when $p\to 0$  as a poly-log of $n$.

\section{A continuous variational problem}\label{sec:cont-var}
In this section we compare the optimum $\phi(n,p,\delta)$ of the variational problem~\eqref{eq:var} with an analogue, $\phi(p,\delta)$, that eliminates the dependence on $n$.
Before introducing this variant, we begin with the straightforward upper bound on $\phi(n, p, \delta)$, which involves constructing $G \in \sG_n$ with $I_p(G)$ that attains the right-hand side of \eqref{eq:var-ans}. There are two competing candidates.
\begin{itemize}
\item Let $g_{ij} = 1$ whenever $1 \leq i < j \leq a$ for some integer $a$ to be
  specified later, and
  $g_{ij} = p$ for all other $i,j$. Then we have
  \[
  t(G) \geq n^{-3}\left[a(a-1)(a-2) + (n(n-1)(n-2) - a(a-1)(a-2)) p^3\right]
  \]
  and
  \[
  I_p(G) = \tbinom{a}{2} I_p(1) = \tbinom{a}{2} \log ( 1/p)\,.
  \]
  So, we can choose $a = (\delta^{1/3} + o(1)) p n$ so that $t(G)
  \geq (1 + \delta)p^3$ and
  \[
  I_p(G) = \bigg(\frac{\delta^{2/3}}{2} + o(1)\bigg)
  n^2 p^2 \log (1/p)\,.
  \]
\item Let $g_{ij} = 1$ whenever $1 \leq i \leq a$ and $i < j$ and
  $g_{ij} = p$ otherwise. Then
  \[
  t(G) \geq n^{-3}\left[3a(n-a)(n-a-1) p + (n-a)(n-a-1)(n-a-2) p^3\right]
  \]
  and
  \[
  I_p(G) = a\bigg(n - \frac{a+1}{2}\bigg) I_p(1) = a\bigg(n -
    \frac{a+1}{2}\bigg) \log(1/p)\,.
  \]
  So, we can choose $a = (\delta/3 + o(1)) p^2 n$ so that $t(G) \geq
  (1 + \delta)p^3$ and
  \[
  I_p(G) =
  \bigg(\frac{\delta}3 + o(1)\bigg)
    n^2 p^2 \log (1/p)\,.
    \]
\end{itemize}
When $p \gg n^{-1/2}$, both constructions are valid, and taking the one
with smaller $I_p(G)$ (the choice depends on the value of $\delta$;
when $\delta \geq 27/8$ we use the first construction and when
$\delta < 27/8$ we use the second construction) yields the upper bound on
$\phi(n,p,\delta)$ in \eqref{eq:var-ans}.

When $n^{-1} \ll p \ll n^{-1/2}$, the second construction is no longer
valid (since $a \ll 1$), but the first construction remains valid. Thus,
we obtain the upper bound on $\phi(n,p,\delta)$ in
\eqref{eq:var-ans-b}.

\medskip

Next, consider the following variant of the above variational problem. Whereas in $\phi(n, p, \phi)$ the variational
problem occurs in the space of weighted graphs on $n$ vertices, in the
new variational problem $\phi(p,\phi)$, we consider the space of
graphons, so that $n$ does not appear (and the dependence of $p$ on
$n$ plays no role). Here a \emph{graphon} is
a symmetric measurable function $W \colon [0,1]^2 \to
[0,1]$. Let $\cW$ denote the set of all graphons.

Given any graphon $W$ and function $f \colon \RR \to \RR$, we use the shorthand notation
\[
\EE[f(W)] := \int_{[0,1]^2}f(W(x,y)) \, dxdy \,.
\]
For example, 
$\EE W^2 = \int_{[0,1]^2} W^2\,dxdy$, and $\EE[I_p(W)] = \int_{[0,1]^2} I_p(W(x,y)) \, dxdy$.

\begin{definition*}
  [Continuous variational problem]
 For $\delta>0$ and $0<p<1$, let
\begin{equation}
  \label{eq:var2}
  \phi(p,\delta) := \inf\left\{ \frac12 \EE[I_p(W)] : W \in \cW \text{ such that } t(W)
  \geq (1 + \delta) p^3 \right\}\,,
\end{equation}
where the triangle density $t(W)$ of $W$ is defined by
\[
t(W) := \int_{[0,1]^3} W(x,y)W(x,z)W(y,z) \, dxdydz \,.
\]
\end{definition*}

The two variational problems \eqref{eq:var} and \eqref{eq:var2} are
related by the following inequality.

\begin{lemma} \label{lem:var-relate}
  For any $p, n, \delta$, we have
  \begin{equation} \label{eq:relate}
    \phi(p, \delta) \leq \frac{1}{n^2} \phi(n, p, \delta) + \frac{1}{2n}I_p(0).
  \end{equation}
\end{lemma}

\begin{proof}
  For any $G \in \sG_n$, we can construct a $W^G \in \cW$ by dividing
  $[0,1]$ into $n$ equal intervals $I_1, \dots, I_n$, and setting
  $W(x, y) = g_{ij}$ whenever $x \in I_i$ and $y \in I_j$. Then $t(W^G)
  =  t(G)$ and $\tfrac12 \EE[I_p(W^G)] = n^{-2} I_p(G) + (2n)^{-1} I_p(0)$,
  where the extra term $(2n)^{-1}I_p(0)$ is due to the zero entries
  $g_{ii} = 0$ which were not included in $I_p(G)$.
\end{proof}

The following theorem, providing a solution to the variational problem $\phi(p,\delta)$, is proved in the next section (see~\S\ref{sec:solve-var'}).

\begin{theorem}
  \label{thm:var2}
  Fix $\delta > 0$. Then
\begin{equation} \label{eq:var2-ans}  \lim_{p \to 0} \frac{\phi(p,\delta)}{p^2 \log(1/p)} =
  \min\bigg\{\frac{\delta^{2/3}}{2} \;,\; \frac{\delta}{3}\bigg\}\,.
\end{equation}
\end{theorem}

It can already be seen that the solution to the variational problem~\eqref{eq:var}
when $n^{-1/2} \ll p \ll 1$ (i.e., Eq.~\eqref{eq:var-ans}) will readily follow from the combination of Lemma~\ref{lem:var-relate} and Theorem~\ref{thm:var2}.
We defer the full details --- together with the treatment of the regime $n^{-1} \ll p \ll n^{-1/2}$ (which will entail a short modification of the proof of Theorem~\ref{thm:var2}) to the next section following the proof of Theorem~\ref{thm:var2} (see~\S\ref{sec:discrete}).

For now, let
us give the constructions that give tight upper bounds on
$\phi(p,\delta)$ for \eqref{eq:var2-ans} --- precisely the
graphon analogs of the above given constructions for Theorem~\ref{thm:var}.
\begin{itemize}
\item Let $W(x,y) = 1$ whenever $x,y \in [0,a]$ for some $a \in [0,1]$
  to be specified later, and $W(x,y) = p$ elsewhere. Then we have
  \[
  t(W) \geq a^3 + (1-a)^3p^3
  \]
  and
  \[
  \frac 12 \EE[I_p(W)] = \frac{1}{2} a^2 I_p(1) = \frac{1}{2} a^2 \log (1/p)\,.
  \]
  So, we can choose $a = (\delta^{1/3} + o(1))p$ so that $t(W) \geq
  (1+\delta)p^3$ and
  \[
  \frac 12 \EE[I_p(W)] = \bigg(\frac{\delta^{2/3}}{2} + o(1)\bigg) p^2 \log(1/p)\,.
  \]
\item Let $W(x,y) = 1$ whenever $\min\{x,y\} \leq a$ and $W(x,y) = p$
  otherwise. Then
  \[
  t(W) \geq 3a(1-a)^2p + (1-a)^3p^3
  \]
  and
  \[
  \frac 12 \EE[I_p(W)] = a\bigg(1 - \frac{a}{2}\bigg) I_p(1) = a\bigg(1 - \frac{a}{2}\bigg)
  \log (1/p)\,.
  \]
  So, we can choose $a = (\delta/3 + o(1)) p^2$ so that $t(G) \geq (1
  + \delta) p^3$ and
  \[
  \frac 12 \EE[I_p(W)] = \paren{\frac{\delta}{3} + o(1)} p^2 \log(1/p)\,.
  \]
\end{itemize}
Depending on the value of $\delta$ (when $\delta \geq 27/8$ use the
first construction; when $\delta < 27/8$ use the second), these two examples together prove the upper bound to $\phi(p,\delta)$
in \eqref{eq:var2-ans}.

\section{Solving the variational problem}\label{sec:solve-var}

\subsection{Proof of Theorem~\ref{thm:var2}}\label{sec:solve-var'}
Throughout this proof, we will occasionally require various technical properties of the function
$I_p$ when $p \to 0$; the proofs of these are deferred to \S\ref{sec:I_p}.

Let $W \in \cW$ satisfy $t(W) \geq (1 + \delta)p^3$. We wish to show
  that
  \begin{equation*} 
  \frac12\EE[I_p(W)] \geq (1-o(1))
  \min\bigg\{\frac{\delta^{2/3}}{2} \;,\; \frac{\delta}{3}\bigg\} p^2
    I_p(1)\,.
\end{equation*}
Since $I_p$ is decreasing in $[0,p]$ and increasing in $[p,1]$, we may
assume without loss of generality that $W \geq p$ and $t(W) =
(1+\delta)p^3$. Write $W = U + p$, so that $0 \leq U \leq 1-p$. Letting
\[
s(U) := \int_{[0,1]^3} U(x,y)U(x,z) \, dxdydz =
\int_{[0,1]} \paren{\int_{[0,1]} U(x,y) \, dy}^2 \, dx\,,
\]
we have
\begin{equation}\label{eq:K3-split}
t(W) - p^3 = t(U) + 3p s(U) + 3p^2 \EE U = \delta p^3\,.
\end{equation}
Now write
\begin{equation*}
  t(U) = \delta_1 p^3, \qquad
  s(U) = \delta_2 p^2, \quad \text{and} \quad
  \EE U = \delta_3 p\,.
\end{equation*}
Then $\delta_1 + 3\delta_2 + 3\delta_3 =
\delta$.
We may assume, for instance, that
\[
\delta_3 \leq \sqrt{p } \log (1/p)= o(1), \quad \text{so that } \EE U = o(p)\,,
\]
since otherwise by the convexity of
$I_p$ and the fact that $I_p(p+x)\sim x^2/(2p)$ for $x \ll p$ (see Lemma~\ref{lem:I_p-asymp} below) we would already have
\[
\EE[I_p(W)] \ge I_p(\EE W) = I_p(p + \EE U) \geq I_p(p+p^{3/2}\log(1/p)) \gg p^2 I_p(1)\,.
\]
The above decomposition reduces the problem to
studying the following:
\begin{equation*}
\phi'(p, \delta_1, \delta_2) := \inf \Big\{\frac12\EE[I_p(p + U)] :\, U \in \cW
\text{ so that } 0 \leq U \leq 1-p, \ t(U) \geq \delta_1 p^3, \text{ and } s(U) \geq
\delta_2 p^2\Big\}\,.
\end{equation*}
The (asymptotic) solution to this variational problem is given by the
following key lemma.
\begin{lemma} \label{lem:divide-conquer}
  Fix $D > 0$. Then
  \[
  \phi'(p,\delta_1,\delta_2) = \bigg(\frac{\delta_1^{2/3}}{2} +
  \delta_2 + o(1)\bigg) p^2  I_p(1)
  \]
  uniformly for all $\delta_1, \delta_2 \in [0,D]$ as $p \to 0$.
\end{lemma}
Assuming Lemma~\ref{lem:divide-conquer}, let us finish
the proof of Theorem~\ref{thm:var2}. We have
\begin{align*}
\frac12 \EE[I_p(W)]
&\geq \min\{\phi'(p,\delta_1,\delta_2) :\, \delta_1 + 3\delta_2
= \delta - o(1)\}
\\
&= (1-o(1))\min\bigg\{\frac{\delta_1^{2/3}}{2} + \delta_2 :\, \delta_1 +
  3\delta_2 = \delta - o(1)\bigg\}  p^2I_p(1)\,.
\end{align*}
Note that if we fix the value of $\delta_1 + 3\delta_2$, then
$\delta_1^{2/3}/2 + \delta_2$ is minimized when one of $\delta_1$ and
$\delta_2$ is set to zero. It follows that
\[
\frac12 \EE[I_p(W)] \geq (1-o(1))\min\bigg\{\frac{\delta^{2/3}}{2}\;,\;
  \frac{\delta}{3}\bigg\} p^2I_p(1)\,.
\]
We have thus established the desired lower bound for $\phi(p,\delta)$ in
Theorem~\ref{thm:var2}, while the upper bound was already given in~\S\ref{sec:cont-var} (immediately after
the statement of the theorem). This completes the proof of the
Theorem~\ref{thm:var2} modulo Lemma~\ref{lem:divide-conquer}.

\medskip

Towards the proof of Lemma~\ref{lem:divide-conquer}, we need the following result, showing how to lower bound $\EE[I_p(p + U)]$ given
$t(U)$.
\begin{lemma}
  \label{lem:I_p-triangle}
  For any $U \in \cW$ with $0 \leq U \leq 1-p$ we have
  \[
  \EE[I_p(p + U)] \geq \paren{1 - o(1)}I_p(1) t(U)^{2/3}\,.
  \]
  where $o(1)$ is some quantity that goes to zero as $p \to 0$.
\end{lemma}

\begin{proof}
For $p=o(1)$ and any $0\leq x \leq 1-p$ one has $I_p(p+x) \geq (1+o(1))x^2 I_p(1)$ (as established in Corollary~\ref{cor:I_p-sq-lower-bd-2} below); thus,
  \begin{align*}
    \EE[I_p(p + U)] &= \int_{[0,1]^2}I_p(p+U(x,y)) \, dxdy \\
    &\geq \paren{1 - o(1)}I_p(1)\int_{[0,1]^2} U(x,y)^2 \, dxdy
    \geq \paren{1 - o(1)} I_p(1) t(U)^{2/3}\,,
  \end{align*}
where we will justify the last inequality using the fact that
\begin{equation}\label{eq:cauchy}
  t(U) \leq \paren{\int_{[0,1]^2} U(x,y)^2 \, dxdy}^{3/2}\quad \mbox{for any $U \in \cW$}\,.
  \end{equation}
Indeed,~\eqref{eq:cauchy} follows from the Cauchy--Schwarz inequality:
  \begin{align*}
    t(U) &= \int_{[0,1]^3} U(x,y)U(x,z)U(y,z) \, dxdydz
         \\
         &\leq \int_{[0,1]^2} \paren{\int_{[0,1]} U(x,y)^2 \,
           dx}^{1/2} \paren{\int_{[0,1]} U(x,z)^2 \, dx}^{1/2} U(y,z)
         \,dydz
  \end{align*}
which, by two more applications of the Cauchy--Schwarz inequality, is at most
  \begin{align*}
         &\int_{[0,1]} \paren{\int_{[0,1]^2} U(x,y)^2 \,
           dx dy}^{1/2} \paren{\int_{[0,1]} U(x,z)^2 \,
           dx}^{1/2} \paren{\int_{[0,1]} U(y,z)^2 \, dy}^{1/2}
         \,dz
         \\
         \leq &\paren{\int_{[0,1]^2} U(x,y)^2 \,
           dx dy}^{1/2} \paren{\int_{[0,1]^2} U(x,z)^2 \,
           dxdz}^{1/2} \paren{\int_{[0,1]^2} U(y,z)^2 \, dydz}^{1/2}\,,
  \end{align*}
as required.
\end{proof}

Lemma~\ref{lem:I_p-triangle} already shows that
$\phi'(p,\delta_1,\delta_2) \geq (\delta_1^{2/3}/2 - o(1)) p^2
I_p(1)$. However, this is not enough. To obtain the additional
$\delta_2p^2I_p(1)$ term in the lower bound of $\phi'$, we isolate the high
degree vertices and consider their contributions.

\begin{proof}[\textbf{\emph{Proof of Lemma~\ref{lem:divide-conquer}}}]
  First we prove an upper bound on $\phi'(p,\delta_1,\delta_2)$. Let
  $A$ be the union of the rectangles \[ [0, \delta_1^{1/3}p]^2\,,\quad
  [0,\delta_2p^2]\x [0,1]\,,\mbox{ and }\quad [0,1] \x [0,\delta_2p^2]\,.\] Set $U$
  to be $1-p$ on $A$ and 0 elsewhere. Then we have $t(U) \geq \delta_2p^3$,
  and $s(U) \geq \delta_2p^2$, whereas $\tfrac12 \EE[I_p(p+U)] = \frac12 \lambda(A)
  I_p(1) = (\frac12 \delta_1^{2/3} + \delta_2 + o(1))p^2I_p(1)$, where here
  and in what follows $\lambda$ denotes Lebesgue measure. This proves the upper bound
  on $\phi'(p,\delta_1,\delta_2)$.

  Assume that $\EE[I_p(p+U)] = O(p^2\log(1/p))$ (with an implicit constant that may
  depend on $D$), or else we are done.

  Let $f(x) = \int_{[0,1]} U(x,y) \, dy$.  Let $b = p^{1/3}$ (any
  choice of $b$ with $\sqrt{p\log(1/p)} \ll b \ll 1$ suffices), and $B
  = \{x \mid f(x) > b\} \subseteq [0,1]$. 
  By the convexity of $I_p$ we have
  \[
  \EE[I_p(p + U)] = \int_{[0,1]^2} I_p(p+U(x,y)) \, dxdy \geq \int_{[0,1]} I_p(p + f(x)) \, dx \geq \lambda(B) I_p(p+b)\,.
  \]
Since $I_p(p+b) = (1 + o(1)) b \log (b/p)$ (see Lemma~\ref{lem:I_p-asymp} below),
  \begin{equation} \label{eq:B-upp-bd}
    \lambda(B) \leq \frac{\EE[I_p(p+U)]}{I_p(p+b)} = \frac{O(p^2 \log
      (1/p))}{(1+o(1))b \log (b/p)} = O\paren{\frac{p^2}{b}}\,.
  \end{equation}
Next,
  we have $I_p(p + x) \geq (x/b)^2 I_p(p + b)$ for $x \in
  [0,b]$ (see Lemma~\ref{lem:I_p-sq-lower-bd-1} below); hence,
  \[
  \EE[I_p(p + U)] \geq \int_{[0,1] \setminus B} I_p(p + f(x)) \, dx \geq
  \frac{I_p(p + b)}{b^2}\int_{[0,1] \setminus B} f(x)^2 \, dx\,.
  \]
Therefore,
  \begin{equation} \label{eq:B-f-sq-upper-bd}
  \int_{[0,1] \setminus B} f(x)^2 \, dx \leq \frac{\EE[I_p(p+U)] b^2}{I_p(p+b)}
  = O(p^2 b)\,,
  \end{equation}
  where the last step is by~\eqref{eq:B-upp-bd}. Since $\int_{[0,1]} f(x)^2 \, dx = s(U) \geq \delta_2 p^2$, we
  have
  \[
  \int_B f(x)^2 \, dx \geq (\delta_2 - O(b))p^2 = (\delta_2 - o(1))p^2\,.
  \]
  First applying the convexity of $I_p$, then
  the fact (shown in Corollary~\ref{cor:I_p-sq-lower-bd-2} below) that $I_p(p+x)$ is at least $(1-o(1))x^2 I_p(1)$ for $p=o(1)$,
  and finally~\eqref{eq:B-f-sq-upper-bd}, we obtain
  \begin{align*}
    \int_{B \x [0,1]} I_p(p + U(x,y)) \, dxdy
    &\geq \int_B I_p(p + f(x)) \, dx
    \\
    &\geq (1 - o(1)) \int_{B} f(x)^2  I_p(1) \, dx
    \geq (\delta_2 - o(1)) p^2I_p(1)\,.
  \end{align*}
  Since $U(x,y) = U(y,x)$, we have
  \begin{equation} \label{eq:I_p-U-K12}
    \frac12 \int_{B \x [0,1] \cup [0,1] \x B} I_p(p + U(x,y)) \, dxdy
  \geq (\delta_2 - o(1)) p^2I_p(1) - \frac12 \lambda(B)^2 I_p(1)
  \geq (\delta_2 - o(1)) p^2I_p(1)\,,
\end{equation}
where the last step is due to $\lambda(B) = O(p^2/b) = o(p)$.

We have $\EE[I_p(p+U)] \geq I_p(p + \EE U) $ by convexity of $I_p$. As
$I_p(p+x)$ is increasing for $x \in [0,1-p]$, and
Lemma~\ref{lem:I_p-asymp} tells us that $I_p(p+ C
p^{3/2}\sqrt{\log(1/p)}) \sim \frac12 C^2 p^2\log(1/p)$ for each fixed
$C > 0$ as $p \to 0$, we see that $\EE[I_p(p+U)] = O(p^2 \log (1/p))$
implies that $ \EE U = O(p^{3/2} \sqrt{\log(1/p)}).  $ Let $U' = U
\mathbf{1}_{B^c \x B^c}$ where $B^c = [0,1] \setminus B$. We have
  \begin{align}
    t(U) - t(U')
    &\leq 3 \int_{B \x [0,1] \x
      [0,1]} U(x,y)U(x,z)U(y,z) \, dxdydz
    \nonumber\\
    &\leq 3 \int_{B \x [0,1] \x
      [0,1]} U(y,z) \, dxdydz
    =3 \lambda(B) \EE U
    = O\paren{b^{-1}p^{7/2}\sqrt{\log(1/p)}} = o(p^3)\,.
    \label{eq-t(U)-t(U')}
  \end{align}
  Thus,
  \[
  t(U') \geq (\delta_1 - o(1)) p^3.
  \]
  By Lemma~\ref{lem:I_p-triangle},
  \begin{multline}\label{eq:I_p-U-K_3}
  \frac12 \int_{B^c \x B^c} I_p(p + U(x,y)) \, dxdy
  = \frac12 \EE[I_p(p + U')] \\
  \geq \bigg(\frac12 - o(1)\bigg) I_p(1) t(U')^{2/3}
  \geq \bigg(\frac{\delta_1^{2/3}}{2} - o(1)\bigg) p^2 I_p(1)\,.
\end{multline}
Combining \eqref{eq:I_p-U-K12} and \eqref{eq:I_p-U-K_3}, we deduce
that
\[
\frac{1}{2} \int_{[0,1]^2} I_p(p+U(x,y)) \, dxdy \geq \bigg(\frac{\delta_1^{2/3}}{2} +
  \delta_2 - o(1)\bigg) p^2  I_p(1)\,.
\]
This proves the lower bound on $\phi'(p,\delta_1, \delta_2)$.
\end{proof}

\subsection{Discrete variational problem --- proof of
  Theorem~\ref{thm:var}} \label{sec:discrete}

First consider the case $n^{-1/2} \ll p \ll 1$.
The upper bound on the left-hand side of~\eqref{eq:var-ans} was already proved in \S\ref{sec:cont-var}. For the lower bound, by applying Lemma~\ref{lem:var-relate} and then Theorem
  \ref{thm:var2} we have
  \[
    \lim_{n \to \infty}\frac{\phi(n, p, \delta)}{n^2 p^2 \log(1/p)}
    \geq \lim_{p \to 0} \frac{\phi(p,\delta)}{p^2\log(1/p)} -
    \lim_{n \to \infty} \frac{I_p(0)}{2n p^2\log(1/p)}
    = \min\bigg\{\frac{\delta^{2/3}}{2}, \frac{\delta}{3}\bigg\} - 0\,.
  \]
  The last zero is due to $I_p(0)/(n p^2\log(1/p)) \sim
  1/(n p\log(1/p)) \to 0$.
  This proves~\eqref{eq:var-ans}.

It remains to treat the regime $n^{-1}
\ll p \ll n^{-1/2}$. When $\delta \geq 27/8$, so
that $\delta^{2/3}/2 \leq \delta/3$, the desired result again follows from
Theorem~\ref{thm:var2} by the same argument as given above. However, when
$\delta < 27/8$, second upper bound construction (stated
immediately following Theorem~\ref{thm:var}) is invalid. In order to prove a matching lower bound for
\eqref{eq:var-ans-b}, we need to eliminate the second
construction as a possibility. We sketch the modifications
to the proof here. It suffices to show that
$s(U) = o(p^2)$ (using the notation of the previous
subsection). Indeed, once we know that $s(U) = o(p^2)$, the
decomposition \eqref{eq:K3-split} implies $t(U) = (\delta -
o(1))p^3$, from which we obtain $\tfrac12 \EE[I_p(p + U)] \geq
(\delta^{2/3}/2 - o(1))p^2I_p(1)$ by
Lemma~\ref{lem:I_p-triangle}.

From now on assume that $n^{-1} \ll p \ll n^{-1/2}$. Assume
$b$ is chosen so that
\[ \max\{p^2n,\sqrt{p \log(1/p)}\} \ll b
\ll 1\,.\]
Then \eqref{eq:B-upp-bd} gives $\lambda(B) =
O(p^2/b) \ll 1/n$. Since we are in the discrete setting of
Theorem~\ref{thm:var}, $\lambda(B) \ll 1/n$ implies that $B$ must be an empty set. Therefore, from
\eqref{eq:B-f-sq-upper-bd} we can infer that $s(U) = \int_{[0,1]} f(x)^2
\,dx = O(p^2b) = o(p^2)$, as claimed. This completes the
proof.
\qed

\subsection{Properties of the function $I_p$ as $p\to 0$} \label{sec:I_p}

Here we collect the various facts about $I_p$ that were referred to throughout the
proof of Theorem~\ref{thm:var2}.

\begin{lemma} \label{lem:I_p-asymp}
  Let $p \to 0$. If $0 \leq x \ll p$, then $I_p(p + x) \sim x^2/(2p)$. If $p
  \ll x \leq 1-p$, then $I_p(p+x) \sim x \log(x/p)$.
\end{lemma}

\begin{proof}
  We use Taylor expansion for $I_p(x)$ around $x = p$,
  noting that $I_p(p)=I_p'(p) = 0$, $I_p''(p) = 1/(p(1-p))$ and $I_p'''(x) = 1/(1-x)^2 - 1/x^2$. We have
  $I_p(p + x) = x^2 I_p''(p)/2 + x^3 I_p'''(\xi)/6$ for some
  $\xi \in (p,p+x)$; thus, $I_p(p+x) = x^2/(2p(1-p)) +
  O(x^3/p^2) \sim x^2/(2p)$ when $0 \leq x \ll p$.

  If $p \ll x < 1-p$ (the required statement trivially holds for $x=1-p$), then
  \begin{equation}
  I_p(p+x) = (p+x) \log \frac{p+x}{p} + (1-p-x) \log\frac{1-p-x}{1-p}
  = (1 + o(1)) x \log \frac{x}{p} + O(x)\,,
\end{equation}
where the bound $O(x)$ comes from $\abs{\log y} \leq y^{-1} - 1$ which is
valid for all $y \in (0,1]$. This shows that $I_p(p+x) \sim
x\log(x/p)$ when $p \ll x \leq 1-p$.
\end{proof}

\begin{lemma} \label{lem:I_p-sq-lower-bd-1}
  There exists $p_0 > 0$ so that for all $0 < p \leq p_0$ and
  $0 \leq x \leq b \leq 1 - p-1/\log (1/p)$,
  \begin{equation}\label{eq:I_p-sq-lower-bd}
  I_p(p + x) \geq (x/b)^2 I_p(p+b)\,.
\end{equation}
\end{lemma}

\begin{proof}
  Let $x_p = 1-p-1/\log(1/p)$. We will show that the
  function $f(x) = I_p(p + \sqrt{x})$ is concave for $x \in [0, x_p^2]$. The inequality \eqref{eq:I_p-sq-lower-bd} then
  follows because for each $b \leq x_p$, the chord joining $(0,0)$ and $(b^2,
  I_p(p+b))$ lies below $f$, so
  that $f(x)\geq (x/b^2) I_p(p+b)$ for all $0 \leq x
  \leq b^2$. Replacing $x$ by $x^2$ yields \eqref{eq:I_p-sq-lower-bd}.

  We have
  \[
  f''(x) = \frac{1}{4(1-p-\sqrt{x})(p+\sqrt{x})x} + \frac{1}{4x^{3/2}}
  \log\paren{\frac{(1-p-\sqrt{x})p}{(p + \sqrt{x})(1-p)}}\,.
  \]
  Let
  \[
  g(x) = 4x^3 f''(x^2) = \frac{x}{(1-p-x)(p+x)} + \log \paren{\frac{(1-p-x)p}{(p+x)(1-p)}}\,.
  \]
  It now suffices to show that $g(x) \leq
  0$ for $x \in [0,x_p]$, which implies that $f$ is concave in $[0,x_p^2]$. We have
  $g(0) = 0$ and
  \begin{align*}
    g(x_p) &= \log(1/p) \frac{1-p-1/\log(1/p)}{1 - 1/\log(1/p)} +
    \log\paren{\frac{p}{\log (1/p)(1-1/\log(1/p))(1-p)}} \\
    &\leq \log(1/p) - \log(1/p) - \log\log (1/p)  + O\left(1/\log(1/p)\right) = -\log\log(1/p) + o(1)\,.
  \end{align*}
  So, we can choose $p_0$ so that $g(x_p) \leq 0$ for all $p
  \leq p_0$. Furthermore, we have
  \[
  g'(x) = \frac{(-1 + 2p + 2x)x}{(1-p-x)^2(p+x)^2}\,.
  \]
  It follows that $g$ is decreasing when $x < 1/2 - p$ and increasing
  when $x > 1/2 - p$. Since $g(0),g(x_p) \leq 0$, we conclude that  $g(x) \leq 0$ for all $x \in [0,x_p]$.
\end{proof}

\begin{corollary}
  \label{cor:I_p-sq-lower-bd-2}
  There is some $p_0 > 0$ so that for all $0 < p \leq p_0$ and all $0
  \leq x \leq 1 - p$ one has
  \begin{equation} \label{eq:I_p-sq-lower-bd-2}
    I_p(p+x) \geq x^2 I_p(1 - 1/\log (1/p)) =(1+o(1)) x^2 I_p(1)
  \end{equation}
  where the $o(1)$-term goes to zero as $p \to 0$.
\end{corollary}

\begin{proof}
  Let $b = 1-p-1/\log (1/p)$. When $0 \leq x \leq b$, the first inequality
  in \eqref{eq:I_p-sq-lower-bd-2} follows from
  Lemma~\ref{lem:I_p-sq-lower-bd-1} since $b < 1$, and when $b < x \leq 1-p$, it
  follows from $I_p(p+x) \geq I_p(p+b) \geq x^2 I_p(p + b)$ since
  $I_p(p+x)$ is increasing for $x \in [0,1-p]$. The last step in
  \eqref{eq:I_p-sq-lower-bd-2} follows from Lemma~\ref{lem:I_p-asymp}.
\end{proof}

\section{Extension to cliques}\label{sec:cliques}

In this section we extend Theorem~\ref{thm:var} and Corollary~\ref{cor:upper-tail}) to upper tails for clique counts.
\begin{definition*}
  [Discrete variational problem for upper tails of $H$-counts]
  Let $H$ be a graph on $k$ vertices. Recall that $\sG_n$ denotes the set of weighted undirected graphs on $n$ vertices with
edge weights in $[0,1]$. The corresponding variational problem for $\delta>0$ and $0<p<1$ is given by
\begin{equation} \label{eq:var-clq}
\phi_H(n,p,\delta) := \inf \Big\{ I_p(G) : G \in \sG_n \text{ with }
t(H,G) \geq (1 + \delta)p^{|E(H)|}\Big\}\,,
\end{equation}
where
\[
t(H,G) := n^{-k} \sum_{1 \leq x_1,\ldots,x_k \leq n} \,\prod_{i j \in E(H)} g_{x_i x_j}
\]
is the probability that a random map $V(H) \to V(G)$ is a graph homomorphism.
\end{definition*}
\begin{theorem}
  \label{thm:var-clq}
Let $K_k$ be the $k$-clique for a fixed $k\geq 3$, and let $\delta > 0$. Then
\[  \lim_{n\to\infty} \frac{\phi_{K_k}(n,p,\delta)}{n^2 p^{k-1}\log(1/p)} = \begin{cases}
    \min\left\{\frac12 \delta^{2/k}\;,\;\delta/k\right\} & \mbox{if }n^{-1/(k-1)} \ll p \ll 1\,,\\
    \frac12 \delta^{2/k}        & \mbox{if }n^{-2/(k-1)} \ll p \ll n^{-1/(k-1)}\,.
  \end{cases}
\]\end{theorem}

Given Theorem~\ref{thm:var-clq}, the analogue of Corollary~\ref{cor:upper-tail} again follows from the new framework of Chatterjee and Dembo,
which establishes (see~\cite[Theorem~1.2]{CD}) that for any fixed $k\geq 3$, the rate function of upper tails for $K_k$ counts in $\cG(n,p)$ is $(1+o(1))\phi_{K_k}(n,p,\delta)$ provided that $p\geq n^{-\alpha}$ for some $\alpha=\alpha(k)>0$ (in particular, any fixed $0<\alpha < (4k^3-8k^2+k+3)^{-1}$ suffices).

\begin{corollary}\label{cor:upper-tail-clq}
For any fixed $k\geq 3$ there exists some $\alpha=\alpha(k)>0$ so the following holds.
For any fixed $\delta>0$, if $n^{-\alpha} \leq p \ll 1$ then
  \begin{equation*}
    \P\left( t(K_k,\cG_{n,p}) \geq (1 +
    \delta)p^{\binom{k}2}\right) = \exp\left[ - (1-o(1))\min\left\{\tfrac12 \delta^{2/k}\,,\, \delta/k\right\}n^2 p^{k-1} \log(1/p)\right]\,.
  \end{equation*}
\end{corollary}

\subsection{Proof of Theorem~\ref{thm:var-clq}}
Let $K_{1,\ell-1}$ be the star on $\ell$ vertices, and let $e(H)$ and $\Delta(H)$ denote the number of edges and maximum degree in $H$, resp.
The proof will follow from the same arguments used to prove Theorem~\ref{thm:var}, once we establish the next lemma.

\begin{lemma}
  \label{lem:K_k-small}
Fix $k\geq 4$ and let $H$ be a non-edgeless $k$-vertex graph other than $K_k$ and $K_{1,k-1}$.
If $U\in\cW$ is a graphon
with $0 \leq U \leq 1-p$ and
$I_p(p+U) \lesssim p^{k-1}\log(1/p)$, then $t(H, U) \ll p^{e(H)}$.
\end{lemma}
Towards the proof of this lemma, we need the following simple claim.
\begin{claim}
  \label{clm:subgraph}
Let $H=(V,E)$ be a nonempty graph on $k\geq 4$ vertices other than $K_k$ and $K_{1,k-1}$.
  Then $H$ has a spanning subgraph $H'=(V,E')$ with $\Delta(H')\leq  2$ and $ e(H') > 2 e(H)/(k-1)  $.
\end{claim}

\begin{proof}
First, we may assume that $\Delta(H) > 2$, since if $\Delta(H)\leq 2$ then $H'=H$ suffices (as $e(H) > 2e(H)/(k-1)$ for $k\geq 4$).
Second, if $H$ is acyclic then $2e(H)/(k-1) \leq 2$, so one can form $H'$ via 2 edges incident to a vertex (recall $\Delta > 2$), along with
another edge if needed (either disjoint or extending that path, recalling $H\neq K_{1,k-1}$). Thus, if we suppose $H$ is a counterexample to the claim with a minimum number of edges, then $H$ must contains a cycle.

Let $C=(v_0,\ldots,v_{\ell-1})$ be a longest cycle of $H$ (so that $v_i v_{i+1}\in E$, indices taken modulo $\ell$).
  Then $\ell < k$, otherwise we could take $E(H')=E(C)$, since $k > 2 e(H) / (k-1)$ for $H \neq K_k$.

  Denote by $\partial C$ the set of edges in $H$ with at least
  one endpoint in $C_\ell$. We claim that $\abs{\partial C} < \ell(k-1)/2$. Indeed, for any $i$, the vertices $v_i$ and
  $v_{i+1}$ cannot have any common
  neighbors outside $C$ (as otherwise a longer cycle can be formed). Hence, every $u\notin C$ can be connected to at most $\lfloor \ell/2\rfloor $ vertices in $C$,
  and unless all $\binom{\ell}2$ potential edges between the vertices of $C$ are present, $\abs{\partial C} < (k-\ell)\lfloor\ell/2\rfloor + \binom{\ell}2 \leq \ell(k-1)/2$.
  On the other hand, if these $\binom{\ell}2$ edges all belong to $H$, then every $u\notin C$ can be connected to at most one vertex in $C$ (otherwise a longer cycle exists), whence $\abs{\partial C} \leq k-\ell + \binom{l}2 < \ell(k-1)/2$ (the last inequality used $2<\ell<k$).

It follows that $e(H) > |\partial C|$, or else $2e(H) / (k-1) < \ell$ and again we can take $E(H')=E(C)$.
Finally, let $H_1 = (V, E(H)\setminus \partial C)$. As established above, $e(H_1) > e(H) - \ell(k-1)/2$, so it would suffice to find a subgraph $H'_1$ of it with $\Delta(H'_1)\leq 2$ and $e(H_1') \geq 2e(H_1)/(k-1)$, to which we can add the cycle $C$ as a separate connected component. Indeed such a subgraph $H'_1$ exists, since $0<e(H_1)<e(H)$ and $H$ was assumed to be a counterexample minimizing $e(H)$.
\end{proof}

\begin{proof}
  [\emph{\textbf{Proof of Lemma~\ref{lem:K_k-small}}}]
  By Corollary~\ref{cor:I_p-sq-lower-bd-2} (as used in the first step in the proof of Lemma~\ref{lem:I_p-triangle}),
$
  \EE [U^2] \leq (1+o(1)) \EE[I_p(p+U)]/I_p(1) \lesssim p^{k-1}
$.
Next, as a consequence of the generalized H\"older's inequality~\cite{Fin92} (see~\cite[Corollary 3.2]{LZ}),
\begin{equation}
  \label{eq-gen-holder}
 t(F,U) \leq \EE \big[U^{d}\big]^{e(F)/d}
\qquad\mbox{for any graph $F$ with $\Delta(F)\leq d$}\,.
\end{equation}
So, by combining these inequalities,
  $t(H',U) \lesssim p^{(k-1)e(H')/2}$ holds for any $H'$ with $\Delta(H')\leq 2$.
  Taking $H'$ as provided by Claim~\ref{clm:subgraph}, we find that
 $t(H,U) \leq t(H',U) \ll p^{e(H)}$, as desired.
\end{proof}

The upper bound of Theorem~\ref{thm:var-clq} on $\phi_{K_k}$ is obtained via the same constructions of~\S\ref{sec:cont-var}, with modified part sizes: a copy of $K_r$ for $r=\delta^{1/k} n p^{(k-1)/2}$ or a copy of $K_{r,n-r}$ for $r=(\delta/k) n p^{k-1}$.
For the lower bound, one decomposes $t(K_k,W)$ as in~\eqref{eq:K3-split}, in which, by Lemma~\ref{lem:K_k-small}, all terms other than $t(K_k,U)$ and $t(K_{1,k-1},U)$ are negligible. The remaining terms, resp.\ analogous to $t(U)$ and $s(U)$ in \S\ref{sec:solve-var}, are treated as in~\S\ref{sec:solve-var} (e.g.,  $\lambda(B)\lesssim p^{k-1}/b $ replaces $\lambda(B)\lesssim p^2/b$ in Lemma~\ref{lem:divide-conquer}) with one exception: instead of~\eqref{eq-t(U)-t(U')}, write
$t(K_k,U)-t(K_k,U') \leq k \lambda(B) t(K_{k-1},U)$; we wish this quantity to be $o\big(p^{\binom{k}2}\big)$, and indeed, since $t(K_{k-1},U) \leq t(H',U) \lesssim p^{(k-1)e(H)/2}$ for any $H'\subset K_{k-1}$ with $\Delta(H')\leq 2$ (as in the proof of Lemma~\ref{lem:K_k-small}), letting $H'=C_{k-1}$ (recall that $k\geq 4$) yields $t(K_{k-1},U) \lesssim p^{(k-1)^2/2}$, and using $\lambda(B) \lesssim p^{k-1}/b$ with $b \gg p^{(k-1)/2}$ completes the proof.
\qed

\subsection{General subgraph counts}
 It is worthwhile noting that the analysis of cliques from the previous section readily implies that, for any fixed graph $F$ with maximum degree $\Delta$,
\begin{equation}
  \label{eq-phiF-order}
  \phi_F(n,p,\delta) \asymp n^2 p^\Delta \log(1/p) \qquad\mbox{whenever } p \gg n^{-1/\Delta}\,.
\end{equation}
Consequently (again via~\cite{CD}), there is some $\alpha=\alpha(F)>0$ such that the rate function $R(n,p,\delta)$ for observing
a number of $F$-copies that is $(1+\delta)$ times its mean in $\cG_{n,p}$ for $p\geq n^{-\alpha}$ is of order $n^2 p^\Delta\log(1/p)$ (the best previous bounds here, cf.~\cite{JOR04}, were $n^2 p^\Delta \lesssim R(n,p,\delta) \lesssim n^2 p^\Delta \log(1/p)$).

\begin{corollary}\label{cor:gen-H}
Let $F$ be a fixed graph with maximum degree $\Delta$. There exist $\alpha=\alpha(F)>0$ such that, for any fixed $\delta>0$ and
any $p\geq n^{-\alpha} $,
  \begin{equation*}
    -\log \P\big( t(F,\cG_{n,p}) \geq (1 +
    \delta)p^{e(F)}\big) \asymp n^2 p^\Delta \log(1/p)\,.
  \end{equation*}
\end{corollary}

Indeed, assume $\Delta \geq 2$ (the case $\Delta=1$ is trivial).  For the upper bound on $\phi_F$ in~\eqref{eq-phiF-order}, take a copy of $K_{r,n-r}$ for $r=\delta n p^{\Delta}$ (as in~\S\ref{sec:cont-var}). For the lower bound, let $W$ be such that $t(F,W)\geq (1+\delta)p^{e(F)}$ and write $U=W-p$ (so $0 \leq U \leq 1-p$).
As in~\eqref{eq:K3-split}, we decompose $t(F,W)-p^{e(F)}$ into $\sum_{H \subseteq F} \theta_{F,H} \, p^{e(F)-e(H)} \, t(H,U)$
for some positive constants $\{\theta_{F,H}\}$,
and by the assumption on $t(F,W)$ there must exist some $H\subseteq F$ with $t(H,U) \gtrsim p^{e(H)}$.
However, by~\eqref{eq-gen-holder},
$
 t(H,U) \leq \EE [U^{\Delta}]^{e(H)/\Delta}
$,
which is at most $\EE[U^2]^{e(H)/\Delta}$ as $\Delta \geq 2$. Combining these,
$ \EE[U^2] \gtrsim p^{\Delta}$,
and yet (by Corollary~\ref{cor:I_p-sq-lower-bd-2}, as before)
$\EE[U^2] \lesssim \EE[I_p(p+U)]/I_p(1)$,
as claimed.

\section{Weak regularity}\label{sec:weak-ldp}

In this section, we give a short proof establishing
\eqref{eq-cd} and Corollary~\ref{cor:upper-tail} for slowly
decreasing $p$, namely $(\log n)^{-1/6} \ll p
\ll 1$, without requiring the new results of Chatterjee and
Dembo.
The lower bound on the tail probability is explained
in the paragraph immediately following
Corollary~\ref{cor:upper-tail}. The upper bound is
established through the following proposition.

\begin{proposition}
  \label{prop:weak-reg-upp-bd}
  Let $0 < \eta <\delta$ and $0 < p < 1$. Then
\begin{equation}\label{eq:union-upper-bound}
  \P(t(\cG_{n,p}) \geq (1+\delta) p^3)
  \leq R \exp\paren{-
    \phi(n,p,\delta-\eta)}\,,
\end{equation}
with $R = M^n \e^{-M^2}$ where $\e = \eta p^3/6 < 1$ and $M = 4^{1/\e^2}$.
\end{proposition}

Assume $\delta > 0$ is fixed and $(\log n)^{-1/6} \ll p \ll
1$. Take a slowly decreasing $\eta = \eta_n$ so that $p^{-3}
(\log n)^{-1/2} \ll \eta \ll 1$. Then $\e = \eta p^3/6 \gg
(\log n)^{-1/2}$, and so, $M = 4^{o(\log n)} = n^{o(1)}$. Thus,
\begin{align*}
\log R = n \log M + M^2 \log (1/\e)
&\ll n \log n +
n^{o(1)}\log\log n \\
&\ll n^2p^2\log(1/p) \asymp
\phi(n,p,\delta) \sim \phi(n,p,\delta - o(1))\,.
\end{align*}
It then follows by Proposition~\ref{prop:weak-reg-upp-bd} that
\[
\P(t(\cG_{n,p}) \geq (1+\delta) p^3) \leq \exp\paren{-(1 -
  o(1)) \phi(n,p,\delta)}\,,
\]
which implies the upper bound in \eqref{eq-cd}. More
generally, one needs $p \gg (\log n)^{-1/(2e(H))}$ in order
to use this method for upper tails of $H$-counts, where $e(H)$ is the number of edges in $H$.

We proceed to prove Proposition~\ref{prop:weak-reg-upp-bd}. Define the relative edge-density between two nonempty subsets of vertices
$A,B\subseteq V(G)$ 
as $d_G(A,B) := |\{(a,b)\in A\times B: ab\in E(G)\}|/(\abs{A}\abs{B})$.

\begin{lemma} \label{lem:single-bound}
  Let $A_1, \dots, A_m$ be a partition of $V = \{1, \dots,
  n\}$ into nonempty sets. Let $\delta>0$, let $0 < p < 1$ and take $0 \leq
  d_{ij} \leq 1$ and $d_{ij} = d_{ji}$ for each $1 \leq
  i, j \leq m$. Suppose that
  \[
  \frac{1}{n^3}\sum_{i,j,k}
  \abs{A_i}\abs{A_j}\abs{A_k}d_{ij}d_{ik}d_{jk} \geq (1 + \delta)p^3\,.
  \]
  Then for a random graph $G \sim \cG_{n,p}$ on
  the vertex set $V$ we have
  \[
      \P (d_G(A_i,A_j)
    \geq
    d_{ij} \text{ for all } 1 \leq i \leq j \leq m)
    \leq \exp\paren{-\phi(n,p,\delta)}\,.
  \]
\end{lemma}

\begin{proof}
  Define $I_p^>(x) := I_p(\max\{x,p\}).$ We know that a
  binomial random variable $X \sim \Bin(N,p)$
  satisfies $\P(X \geq \delta N) \leq \exp(- N
  I_p^>(\delta))$.  We have
\begin{align*}
  \P\left(d_G(A_i,A_j\right) \geq
  d_{ij}) &\leq
  \exp\paren{-\abs{A_i}\abs{A_j} I_p^>(d_{ij})}
  \quad \text{if }
  i \neq j\,,
\\
  \P\left(d_G(A_i,A_i) \geq d_{ii}\right) &\leq
  \exp\paren{-\tbinom{\abs{A_i}}{2} I_p^>(d_{ii})}\,.
\end{align*}
  Let
  \[
  I_p^>(A,d) := \sum_{1 \leq i< j\leq m}
      \abs{A_i}\abs{A_j} I_p^>(d_{ij}) + \sum_{i=1}^m
      \tbinom{\abs{A_i}}{2} I_p^>(d_{ii})\,.
  \]
  Since $A_1, \dots, A_m$ are disjoint, we have
  \[
    \P\left(d_G(A_i,A_j)
    \geq
    d_{ij} \text{ for all } 1 \leq i \leq j \leq m\right)
    \leq
     \exp\paren{-I_p^>(A,d)}
    \leq \exp\paren{-\phi(n,p,\delta)}\,,
  \]
  where the last step follows from the following
  observation: if $G' \in \cG_n$ is the weighted graph on
  vertex set $V$ obtained by setting $g_{xy} =
  \max\{d_{ij},p\}$ whenever $x \in A_i$ and $y \in A_j$, then $t(G') \geq (1+\delta)p^3n^3$, so
  that $I_p^>(A,d) = I_p(G') \geq \phi(n,p,\delta)$ by our
  definition \eqref{eq:var} of $\phi$.
\end{proof}

The following lemma is a consequence of the
Frieze--Kannan weak regularity lemma and an associated
counting lemma (see \cite[\S9.1, \S10.5]{Lov12}).

\begin{lemma}
  \label{lem:FK}
  Let $\e > 0$ and let $G$ be a graph with $n$
  vertices. Then there exists a partition $\cP$ of the
  vertices of $G$ into at most $4^{1/\e^2}$ parts $A_1,
  \dots, A_m$ so that if $d_{ij} = d_G(A_i, A_j)$, then
  \[
  \bigg|t(G) - n^{-3}\sum_{i,j,k=1}^m
    \abs{A_i}\abs{A_j}\abs{A_k}d_{ij}d_{ik}d_{jk}\bigg| \leq 3 \e
    \,.
  \]
\end{lemma}

\begin{proof}[\textbf{\emph{Proof of Proposition~\ref{prop:weak-reg-upp-bd}}}]
  Let  $G$ be any graph on $n$ vertices satisfying $t(G)
  \geq (1+\delta) p^3$. By Lemma~\ref{lem:FK}, there
  exists a partition of its vertices into $m \leq
  M$ parts $A_1, A_2, \dots, A_m$, so that
  \[
  n^{-3}\sum_{i,j,k=1}^m
     \abs{A_i}\abs{A_j}\abs{A_k}d_{ij}d_{ik}d_{jk}
     \geq (1+\delta) p^3 - 3\e\,,
  \]
  where $d_{ij} = d_G(A_i,A_j)$.
  Let $d'_{ij}$ be $d_{ij}$ rounded down to the nearest multiple of
  $\e$. Then
  \begin{equation*}
  n^{-3}\sum_{i,j,k=1}^m
     \abs{A_i}\abs{A_j}\abs{A_k}d'_{ij}d'_{ik}d'_{jk}
     \geq (1+\delta) p^3 - 6\e = (1 + \delta - \eta) p^3\,.
  \end{equation*}
  For any fixed choice of $\{A_i\}_i$,
  $\{d'_{ij}\}_{i,j}$, by Lemma~\ref{lem:single-bound} we have
  \[
  \P (d_G(A_i,A_j)
    \geq
    d'_{ij} \text{ for all } 1 \leq i \leq j \leq m)
    \leq \exp\paren{-\phi(n,p,\delta-\eta)}\,.
    \]
  A union bound over
  the $A_i$'s ($\leq M^n$
  choices) and $d'_{ij}$'s ($\leq \e^{-M^2}$ choices)
  now yields~\eqref{eq:union-upper-bound}.
\end{proof}


\end{document}